\documentclass[11pt,twoside]{amsart}
\usepackage{mathrsfs}
\usepackage{amsmath}
\usepackage{amsthm}

\usepackage{amsfonts}
\usepackage{amssymb}
\usepackage{latexsym}
\usepackage[all]{xy}

\allowdisplaybreaks[4] \footskip=15pt
\renewcommand{\uppercasenonmath}[1]{}

\numberwithin{equation}{section} \theoremstyle{plain}
\newtheorem*{thm*}{Main Theorem}
\newtheorem{theorem}{Theorem}[section]
\newtheorem{corollary}[theorem]{Corollary}
\newtheorem*{corollary*}{Corollary}
\newtheorem{lemma}[theorem]{Lemma}
\newtheorem*{lemma*}{Lemma}
\newtheorem{proposition}[theorem]{Proposition}
\newtheorem*{proposition*}{Proposition}

\newtheorem*{remark*}{Remark}
\newtheorem{definition}[theorem]{Definition}
\newtheorem*{definition*}{Definition}

\newtheorem*{example*}{Example}

\newtheorem*{acknowledgements*}{ACKNOWLEDGEMENTS}

\newcommand{\Ext}{\mbox{\rm Ext}}
\newcommand{\Hom}{\mbox{\rm Hom}}
\newcommand{\Tor}{\mbox{\rm Tor}}
\newcommand{\im}{\mbox{\rm Im}}
\newcommand{\Coker}{\mbox{\rm Coker}}
\newcommand{\Ker}{\mbox{\rm Ker}}
\newcommand{\gldim}{\mbox{\rm gl.dim}}
\newcommand{\id}{\mbox{\rm id}}
\newcommand{\pd}{\mbox{\rm pd}}
\newcommand{\fd}{\mbox{\rm fd}}
\newcommand{\add}{\mbox{\rm add}}

\begin{document}
\begin{center}
{\large  \bf
Gorenstein (semi)hereditary rings with respect to a semidualizing module}
\\ \vspace{0.6cm} {GUOQIANG ZHAO\footnote {Corresponding author} and JUXIANG SUN}
\end{center}

\bigskip
\begin{center}{ \bf  Abstract}
\end{center}

Let $C$ be a semidualizing module. We first investigate the properties of
finitely generated $G_C$-projective modules.
Then, relative to $C$, we introduce and study the rings of
Gorenstein (weak) global dimensions at most 1,
which we call $C$-Gorenstein (semi)hereditary rings,
and prove that every $C$-Gorenstein hereditary ring is both coherent and $C$-Gorenstein semihereditary.
\vspace{0.2cm}


\bigskip


\section { \bf Introduction}
As a nice generation of Gorenstein homological dimensions [4], 
White [9] introduced the $G_C$-projective, injective and flat dimensions,
where $C$ is a semidualizing module.
The author showed that they share many common properties with
the Gorenstein homological dimensions,
which have been extensively studied in recent decades.
However, the proofs may become more complicated.

It is well-known that, the classical global dimensions of rings
play an important role in the theory of rings.
Motivated by Bennis and Mahadou\rq s [1] ideas to study the global dimensions of
a ring $R$ in terms of Gorenstein homological dimensions,
recently, Zhao and Sun [10] studied the global dimensions
of $R$ with respect to a semidualizing module $C$,
and proved that sup$\{G_{C}$-$\pd_R(M)| M$ is an $R$-module$\}$
 = sup$\{G_{C}$-$\id_R(M)| M$ is an $R$-module$\}$.
The common value, denoted by $\mathrm{G}_C$-$\gldim(R)$,
is called the $C$-Gorenstein global dimension of $R$.
Similarly, the $C$-Gorenstein weak
global dimension of $R$ is also defined as
$\mathrm{G}_C$-$w\gldim(R)$ = sup$\{G_C$-$\fd_R(M) |M$ is an $R$-module$\}$.


On the other hand, in classical homological algebra, 
the rings of (weak) global dimesions at most 1, called (semi)hereditary [7], 
are paid more attention, and the following are well-known:
(1) every hereditary ring is coherent and semihereditary;
(2) a ring $R$ is semihereditary if and only if every finitely
generated submodule of a projective module is projective.
Therefore, it is interesting to consider the following questions.

Question A. Is it true that every $C$-Gorenstein hereditary ring is coherent
and $C$-Gorenstein semihereditary?

Question B. Is it true that $R$ is $C$-Gorenstein semihereditary if and only if every finitely
generated submodule of a $G_C$-projective module is $G_C$-projective?

Based on the results mentioned above, in this paper we mainly study the
properties of $C$-Gorenstein (semi)hereditary rings and
discuss the two questions A and B.

\vspace{3mm}


\section { \bf Preliminaries}

Throughout this work $R$ is a commutative ring with unity. For an $R$-module $T$,
let $\add_R T$ be the subclass of $R$-modules consisting of all
modules isomorphic to direct summands of finite direct sums of copies of $T$.

We define gen$^{*}(T)$ = $\{M$ is an $R$-module $|$ there exists an
exact sequence $\cdots\rightarrow T_i\rightarrow\cdots\rightarrow
T_1\rightarrow T_0\rightarrow M\rightarrow 0$ with $T_i\in\add_RT$
and $\Hom_R(T, -)$ leaves it exact $\}$ (see [8]).
cogen$^{*}(T)$ is defined dually.

Semidualizing modules, defined next, form the basis for our categories of interest.

\begin{definition} {\rm ([9])}
An $R$-module $C$ is {\it semidualizing} if

$(1)$ $C$ admits a degreewise finite projective resolution,

$(2)$ The natural homothety map $R\rightarrow\Hom_R(C,C)$ is an isomorphism, and

$(3)$ $\Ext^i_R (C,C) = 0$ for any $i\geq 1$.
\end{definition}

In the following, we always assume that $C$ is a semidualizing $R$-module.

\begin{definition} {\rm ([6])}
An $R$-module is called {\it
$C$-projective} if it has the form $C\otimes_R P$ for some
projective $R$-module $P$. An $R$-module is called {\it
$C$-injective} if it has the form $\Hom_R(C, I)$ for some injective
$R$-module $I$.
\end{definition}


\begin{definition} {\rm ([9])}
An $R$-module $M$ is called {\it $G_C$-projective} if 
there exists an exact sequence of $R$-modules
$$
\mathbb{X}= \cdots\rightarrow P_1\rightarrow
P_0\rightarrow C\otimes_RP_{-1}\rightarrow C\otimes_RP_{-2}\rightarrow \cdots
$$
with all $P_i$ projective, such that $M\cong \im(P_0\rightarrow C\otimes_RP_{-1})$
and $\Hom_R(\mathbb{X}, C\otimes_R P)$ is exact for any projective $R$-module $P$.
\end{definition}

$G_C$-injective modules are defined in a dual manner, 
and denoted by $\mathcal{GP}_C(R)$ (respectively $\mathcal{GI}_C(R)$)
the class of $G_C$-projective (respectively injective) $R$-modules.


\begin{definition} {\rm ([9])}
Let $\mathcal{X}$ be a class of $R$-modules and
$M$ an $R$-module. An {\it $\mathcal{X}$-resolution} of $M$ is an
exact sequence of $R$-modules as follows:
$$\cdots\rightarrow
X_n\rightarrow\cdots\rightarrow X_1\rightarrow X_0\rightarrow
M\rightarrow 0$$ with each $X_i$ $\in$ $\mathcal{X}$ for any $i\geq
0$. The $\mathcal{X}$-${\it projective\ dimension}$ of $M$ is the
quantity
\begin{center}
$\mathcal{X}$-$\pd_{R}(M) = \inf\{\sup \{n\geq0 | X_n\neq0\} | 
X$ is an $\mathcal{X}$-resolution of $M\}$.
\end{center}
\end{definition}

The {\it $\mathcal{X}$-coresolution} and $\mathcal{X}$-${\it
injective\ dimension}$ of $M$ are defined dually. For simplicity,
we write $G_{C}$-$\pd_R(M)$ = $\mathcal{GP}_C(R)$-$\pd_R(M)$
and $G_{C}$-$\id_R(M)$ = $\mathcal{GI}_C(R)$-$\id_R(M)$.

\begin{definition}{\rm ([5])}
$M$ is called {\it $G_C$-flat} if
there is an exact sequence of $R$-modules
$$ \mathbb{Y}= \cdots\rightarrow F_1\rightarrow
F_0\rightarrow C\otimes_RF_{-1}\rightarrow C\otimes_RF_{-2}\rightarrow \cdots$$
with all $F_i$ flat, such that $M\cong \im(F_0\rightarrow C\otimes_RF_{-1})$
and $\Hom_R(C, I)\otimes_R\mathbb{Y}$ is still exact for any injective $R$-module $I$.
We define $G_{C}$-$\fd_R(M)$ analogously to $G_{C}$-$\pd_R(M)$.
\end{definition}

\begin{definition}
A ring $R$ is called $C$-Gorenstein hereditary ($G_C$-hereditary for
short) if every submodule of a projective $R$-module is $G_C$-projective
(i.e. $G_C$-$\gldim(R)\leq 1$), and $R$ is said to be $C$-Gorenstein semihereditary
($G_C$-semihereditary for short) if
$R$ is coherent and every submodule of a flat $R$-module is $G_C$-flat.
\end{definition}


\section { \bf $C$-Gorenstein (semi)hereditary rings}

We use $(-)^C$ to denote the functor $\Hom_R(-, C)$.

\begin{lemma}
Assume that $M$ is a finitely generated $G_C$-projective $R$-module. Then

(1) $M\in$ cogen$^{*}(C)$,

(2) $M^C\in$ gen$^{*}(R)$.
\end{lemma}

\begin{proof}
(1) Because $M$ is $G_C$-projective, there is an exact sequence
$0 \to M\to C\otimes_R F \to G\to 0$, in which $F$ is free
and $G$ is $G_C$-projective by [9, Observation 2.3 and Proposition 2.9].
Since $M$ is also finitely generated, there exist a finitely generated
free submodule $R^{\alpha_0}$ with ${\alpha_0}$ an integer,
and a free submodule $F^\prime$ of $F$,
such that $M\subseteq C\otimes_R R^{\alpha_0}$ and $F=R^{\alpha_0}\oplus F^\prime$.
Setting $H=\Coker( M\to C\otimes_R R^{\alpha_0})$ yields an exact sequence
$$0 \to M\to C\otimes_R R^{\alpha_0} \to H\to 0$$
with $C\otimes_R R^{\alpha_0}\cong C^{\alpha_0}$ $\in\add_RC$
and $H\oplus (C\otimes_R F^\prime)\cong G$.
From [9, Theorem 2.8], we know that $H$ is finitely generated $G_C$-projective.
Repeating this process to $H$ and so on, one has an exact sequence
$$0\rightarrow M\rightarrow C^{\alpha_0}\rightarrow C^{\alpha_1}\rightarrow \cdots\eqno{(\ast)}$$
with each image finitely generated $G_C$-projective,
which implies the sequence $(\ast)$ is exact after applying $\Hom_R(-, C)$.
Therefore, $M\in$ cogen$^{*}(C)$.

(2) From (1), applying $\Hom_R(-, C)$ to $(\ast)$ provides an exact sequence
$$
\cdots\rightarrow (C^{\alpha_1})^C\rightarrow (C^{\alpha_0})^C
\rightarrow M^C\rightarrow 0
$$
Since $(C^{\alpha_i})^C\cong\Hom_R(C, C)^{\alpha_i}\cong R^{\alpha_i}$,
the desired result follows.
\end{proof}


The following theorem plays a crucial role in proving the main result in this paper.

\begin{theorem} A ring R is coherent if every finitely generated submodule of a
$G_C$-projective $R$-module is $G_C$-projective.
\end{theorem}

\begin{proof}
Let $M$ be a finitely generated submodule of a projective $R$-module.
By the hypothesis, $M$ is $G_C$-projective since every projective $R$-module
is $G_C$-projective from [9, Proposition 2.6].
It follows from Lemma 3.1 (2) that $M^C\in$ gen$^{*}(R)$,
and hence it is finitely generated.
On the other hand, since $M$ is finitely generated,
there is an exact sequence
$$0 \to K\to F_0 \to M\to 0,$$ 
where $F_0=R^\alpha$ is finitely generated free.
Applying $\Hom_R(-, C)$ to this short exact sequence gives rise to 
the exactness of $0\rightarrow M^C\rightarrow (R^{\alpha})^C$.
Since $(R^{\alpha})^C=\Hom_R(R^{\alpha}, C)$ $\cong$ $C^{\alpha}$ 
is $G_C$-projective by [9, Proposition 2.6] again, 
the assumption yields that $M^C$ is $G_C$-projective.
Replacing $M$ with $M^C$ in Lemma 3.1, 
we get that $M^{CC}\in$ gen$^{*}(R)$, and hence finitely presented.

From Lemma 3.1 (1), we know that $M\in$ cogen$^{*}(C)$.
Consider the following commutative diagram with exact rows:
$$\begin{array}{ccccccccc}
0 & \rightarrow & M & \rightarrow & C^{\alpha_0}& 
\rightarrow& C^{\alpha_1} &\rightarrow & \cdots \\
& & \downarrow &   & \downarrow  &   & \downarrow &   &  \\
0 & \rightarrow & M^{CC} & \rightarrow & (C^{\alpha_0})^{CC} & 
\rightarrow& (C^{\alpha_1})^{CC} & \rightarrow & \cdots
\end{array}$$
Because $(C^{\alpha_i})^{CC}\cong (R^{\alpha_i})^{C}\cong C^{\alpha_i}$
for each $i\geq 0$, $M\cong M^{CC}$, and consequently finitely presented.
Thus, every finitely generated ideal of $R$ is finitely presented,
and so $R$ is coherent.
\end{proof}


To prove the coherence of $G_C$-hereditary rings, we need the following result,
which gives some other descriptions of $G_C$-hereditary rings.

\begin{proposition}
Let $R$ be a ring, the following are equivalent.

(1) $R$ is $G_C$-hereditary.

(2) Every submodule of a $G_C$-projective $R$-module is $G_C$-projective.

(3) Every quotient module of a $G_C$-injective $R$-module is $G_C$-injective.
\end{proposition}

\begin{proof}
$(1)\Rightarrow (2)$ 
Let $M$ be a submodule of a $G_C$-projective $R$-module $G$,
one has an exact sequence
$0\rightarrow M\rightarrow G\rightarrow G/M\rightarrow 0$.
Because $R$ is $G_C$-hereditary, $G_C$-$\pd_R(G/M)\leq 1$,
and so $M$ is $G_C$-projective by [9, Proposition 2.12].

$(2)\Rightarrow (1)$ Suppose that $M$ is an $R$-module.
There is an exact sequence
$0\rightarrow K\rightarrow P\rightarrow M\rightarrow 0$
with $P$ projective, and $K=\Ker(P\rightarrow M)$.
Because every projective $R$-module is $G_C$-projective,
the hypothesis yields that $K$ is $G_C$-projective.
Thus $G_C$-$\pd_R(M)\leq 1$, as desired.

$(2)\Leftrightarrow (3)$ Follows from [10, Theorem 4.4].
\end{proof}

\begin{corollary}
Every $G_C$-hereditary ring is coherent.
\end{corollary}

\begin{proof}
The conclusion follows from Theorem 3.2 and Proposition 3.3.
\end{proof}

In the special case $C=R$, we obtain the main result of [3, Theorem 2.5].

\begin{corollary}
All Gorenstein hereditary rings are coherent.
\end{corollary}

Before starting to study the $G_C$-semihereditary rings, 
we first give some equivalent characterizations of $G_C$-flat dimension.
\begin{lemma}
Assume that $R$ is a coherent ring, and $M$ an $R$-module with $G_C$-$\fd_R(M)<\infty$.
For an nonnegative integer $n$, the following are equivalent.

(1) $G_C$-$\fd_R(M)\leq n$.

(2) $\Tor_{i>n}^R(M, \Hom_R(C, I)) = 0$ for any injective module $I$.

(3) In every exact sequence $0\to K_n\to G_{n-}1\to\cdots\to G_0\to M\to 0$, where the
$G_i$ are $G_C$-flat, one has that $K_n$ is also $G_C$-flat.
\end{lemma}

\begin{proof}
The conclusion follows from [11, Theorem 3.8] and
the dual version of [9, Proposition 2.12].
\end{proof}

\begin{lemma}
Let $R$ be a ring, the following are equivalent.

(1) $R$ is $G_C$-semihereditary.

(2) $R$ is coherent and $G_C$-$w\gldim(R)\leq 1$.

(3) R is coherent and every submodule of a $G_C$-flat $R$-module is $G_C$-flat.
\end{lemma}

\begin{proof}
$(1)\Rightarrow (2)$ Let $M$ be an $R$-module.
There is an exact sequence
$0\rightarrow K\rightarrow P\rightarrow M\rightarrow 0$
with $P$ projective. By the hypothesis, $K$ is $G_C$-flat.
Because $R$ is coherent, every flat $R$-module is $G_C$-flat from [11, Corollary 3.9].
Thus $G_C$-$\fd_R(M)\leq 1$, as desired.

$(2)\Rightarrow (3)$ Let $M$ be a submodule of a $G_C$-flat module $G$.
Consider the exact sequence
$0\rightarrow M\rightarrow G\rightarrow G/M\rightarrow 0$.
The assumption yields $G_C$-$\fd_R(G/M)\leq 1$,
and so $M$ is $G_C$-flat by Lemma 3.6.

$(3)\Rightarrow (1)$ Because $R$ is coherent,
every flat $R$-module is $G_C$-flat from [11, Corollary 3.9].
The assertion obviously holds.
\end{proof}

The next result, together with Corollary 3.4, 
gives an affirmative answer to Question A.

\begin{theorem}
If $R$ is a $G_C$-hereditary ring, then it is $G_C$-semihereditary.
\end{theorem}

\begin{proof}
The coherence of $R$ follows from Corollary 3.4.
Because $\mathrm{G}_C$-$w\gldim(R)$ $\leq$ 
$\mathrm{G}_C$-$\gldim(R)\leq 1$ by [10, Corollary 4.6],
Lemma 3.7 implies that $R$ is $G_C$-semihereditary.
\end{proof}

Recall that an $R$-module $M$ is called $FP$-injective if $\Ext^1_R(N, M) = 0$ for all
finitely presented modules $N$. The $FP$-injective dimension of $M$, denoted
by $FP$-$\id(M)$, is defined to be the least nonnegative integer $n$ such that
$\Ext^{n+1}_R(N, M) = 0$ for all finitely presented modules $N$.
If no such $n$ exists, set $FP$-$\id(M)=\infty$.

\begin{lemma}
Let $R$ be a coherent ring, the following are equivalent.

(1) $FP$-$\id(C\otimes_R P)\leq n$ for any projective module $P$,

(2) $\fd(\Hom_R (C, I))\leq n$ for any injective module $I$.
\end{lemma}

\begin{proof}
We use $(-)^+$ to denote the functor $\Hom_R(-, E)$,
where $E$ is an injective cogenerator for the category of $R$-modules.

$(1)\Rightarrow (2)$ Let $P$ be a projective $R$-module.
The hypothesis implies that $\Ext^{i>n}_R(N, C\otimes_R P) = 0$
for all finitely presented modules $N$.
Because $\Tor_{i>n}^R((C\otimes_R P)^+, N)$ $\cong$
$(\Ext^{i>n}_R(N$, $C\otimes_R$ $P))^+$ = 0 by [2, Proposition 5.3],
$\fd(C\otimes_R P)^+\leq n$, and so $\fd(\Hom_R(C, P^+))\leq n$.

For any injective module $I$,
since $P^+$ is also an injective cogenerator,
$I$ is a direct summands of $\prod P^+$.
Thus $\Hom_R (C, I)$ is a direct summands of
$\Hom_R (C, \prod P^+)$ $\cong$ $\prod\Hom_R (C, P^+)$,
and so $\fd(\Hom_R (C, I))\leq$ $\fd(\prod\Hom_R (C, P^+))\leq n$
from the coherence of $R$.

$(2)\Rightarrow (1)$ Suppose $F$ is a flat $R$-module, then
$F^+$ is injective, and so $\fd(\Hom_R (C, F^+))\leq n$.
The adjoint isomorphism $\Hom_R(C, F^+)\cong (C\otimes_R F)^+$
yields $\fd(C\otimes_R F)^+\leq n$.
Thus, for any $R$-module $N$,
$(\Ext^{i>n}_R(N, C\otimes_R F))^+$ $\cong$
$\Tor_{i>n}^R((C\otimes_R F)^+, N) = 0$, and hence
$\Ext^{i>n}_R(N, C\otimes_R F)=0$. Therefore,
$\id(C\otimes_R F)\leq n$, and so 
$FP$-$\id(C\otimes_R P)$ $\leq$ $\id(C\otimes_R P)\leq n$
for any projective module $P$.
\end{proof}

The following result gives a partial answer to Question B.

\begin{theorem} Let $R$ be a ring with $\mathrm{G}_C$-$w\gldim(R)<\infty$.
If every finitely generated submodule of a $G_C$-projective $R$-module is $G_C$-projective,
then $R$ is $G_C$-semihereditary.
\end{theorem}

\begin{proof}
Let $M$ be a finitely presented $R$-module, there is an exact sequence
$$0\rightarrow K\rightarrow P\rightarrow M\rightarrow 0$$
with $P$ finitely generated projective.
Note that $R$ is coherent by Theorem 3.2,
it follows that $K$ is finitely generated.
By the hypothesis, $K$ is $G_C$-projective.
Thus, one has $G_C$-$\pd_R(M)\leq 1$ for
every finitely presented $R$-module $M$,
and so $\Ext^{i>1}_R(M, C\otimes_R Q)=0$ 
for any projective module $Q$ by [9, Proposition 2.12].
This means that $FP$-$\id(C\otimes_R Q)\leq 1$.
By Lemma 3.9, $\fd(\Hom_R (C, I))\leq 1$ for any injective module $I$,
which implies that $\Tor_{i>1}^R(N, \Hom_R(C, I))$ = 0 for any $R$-module $N$.
Since $G_C$-$\fd(N)<\infty$, Lemma 3.6 yields $G_C$-$\fd(N)\leq 1$.
Therefore, $\mathrm{G}_C$-$w\gldim(R)\leq 1$, and hence
$R$ is $G_C$-semihereditary by Lemma 3.7.

\end{proof}

\vspace{0.5cm}

{\bf ACKNOWLEDGEMENTS.} This research was partially supported by the National
Natural Science Foundation of China (Grant Nos. 11201220, 11401147).

\vspace{0.5cm}

GUOQIANG ZHAO, Department of Mathematics,

Hangzhou Dianzi University,
Hangzhou, 310018, PR China

e-mail: gqzhao@hdu.edu.cn

\vspace{0.2cm}

JUXIANG SUN, School of Mathematics and
Information Science,

Shangqiu Normal University, Shangqiu, 476000, PR China

e-mail: sunjx8078@163.com

\vspace{0.5cm}

\begin{thebibliography}{s2}

\bibitem{1} D. Bennis and N. Mahdou, {\it Global Gorenstein dimensions}, Proc. Amer. Math. Soc. {\bf 138} (2010), 461--465.


\bibitem{2} H. Cartan and S. Eilenberg, Homological Algebra, Reprint of the 1956 original, Princeton Landmarks in Math.,
Princeton: Princeton University Press, 1999.

\bibitem{3} Z. H. Gao and F. G. Wang, {\it All Gorenstein hereditary rings are coherent},
J. Algebra Appl. {\bf 13} (2014), 1350140 (5 pages).

\bibitem{4} H. Holm, {\it Gorenstein homological dimensions}, J. Pure Appl. Algebra {\bf 189} (2004), 167--193.

\bibitem{5} H. Holm and P. J$\phi$rgensen, {\it Semi-dualizing modules and related Gorenstein homological dimensions}, 
J. Pure Appl. Algebra {\bf 205} (2006), 423--445.


\bibitem{6} H. Holm and D. White, {\it Foxby equivalence over associative rings}, 
J. Math. Kyoto Univ. {\bf 47} (2007), 781--808.

\bibitem{7} J. J. Rotman, An Introductions to Homological Algebra, Academic Press, New York,
1979.

\bibitem{8} T. Wakamatsu, {\it Tilting modules and Auslander¡¯s Gorenstein property}, 
J. Algebra {\bf 275} (2004), 3--39.

\bibitem{9} D. White, {\it Gorensten projective dimension with respect to a semidualizing module}, 
J. Commut. Algebra {\bf 2} (2010), 111--137.

\bibitem{10} G. Q. Zhao and J. X. Sun, {\it Global dimensions of rings with respect to a semidualizing module(preprint)},
available from the arXiv:1307.0628 [math.RA].

\bibitem{11} G. Q. Zhao and X. G. Yan, {\it Resolutions and stability of $C$-Gorenstein flat modules},
accepted for publication in Rocky Mountain Journal of Mathematics.
\end{thebibliography}
\end{document}